\documentclass[11pt]{amsart}

\usepackage[foot]{amsaddr}

\RequirePackage[utf8]{inputenc}
\RequirePackage[T1]{fontenc}
\RequirePackage[english]{babel}
\RequirePackage{amsmath, amssymb, amsthm}

\usepackage{amsxtra}
\usepackage{mathtools}
\usepackage{stmaryrd}
\usepackage[all]{xy}

\RequirePackage{pgf,tikz}
\usetikzlibrary{arrows,decorations.pathmorphing,decorations.pathreplacing,backgrounds,positioning,fit,petri,calc,cd}
\usepackage{caption}

\usepackage{hyperref}
\hypersetup{
    colorlinks  =   true,
    linkcolor   =   blue,
    filecolor   =   blue,      
    urlcolor    =   blue,
    citecolor   =   blue,
}

\theoremstyle{plain}
	\newtheorem{theorem}{Theorem}[section]
	\newtheorem{proposition}[theorem]{Proposition}
	\newtheorem{lemma}[theorem]{Lemma}
	\newtheorem{cor}[theorem]{Corollary}

\theoremstyle{definition}
	\newtheorem{definition}[theorem]{Definition}
	\newtheorem{example}[theorem]{Example}
	\newtheorem{remark}[theorem]{Remark}

\newcommand{\cF}{\mathcal{F}}
\newcommand{\bbC}{\mathbb C}
\newcommand{\bbP}{\mathbb P}
\newcommand{\bbZ}{\mathbb Z}
\newcommand{\bbQ}{\mathbb Q}
\newcommand{\bbR}{\mathbb R}

\title[The perfect cone compactification]{The perfect cone compactification of quotients of type IV domains}
\author{Luca Giovenzana}

\address{Department of Mathematical Sciences, Loughborough University,}
\email{l.giovenzana@lboro.ac.uk}

\begin{document}
\thispagestyle{empty}

\begin{abstract}
\begin{sloppypar}
The perfect cone compactification is a toroidal compactification which can be defined for locally symmetric varieties.
Let $\overline{D_{L}/\widetilde{O}^{+}(L)}^{p}$ be the perfect cone compactification of the quotient of the type IV domain $D_{L}$ associated to an even lattice $L$. In our main theorem we show that the pair ${ (\overline{D_{L}/\widetilde{O}^{+}(L)}^{p}, \Delta/2) }$ has klt singularities, where $\Delta$ is the closure of the branch divisor of ${ D_{L}/\widetilde{O}^{+}(L) }$.
\end{sloppypar}
In particular this applies to the perfect cone compactification of the moduli space of {$2d$-polarised} $K3$ surfaces with ADE singularities when $d$ is square-free.
\end{abstract}

\makeatletter
\@namedef{subjclassname@2020}{
	\textup{2020} Mathematics Subject Classification}
\makeatother

%\keywords{K3 surfaces, toroidal compactifications, perfect cones}

\maketitle

\section{Introduction}
In \cite{SB06} and \cite{SBA16} Shepherd-Barron and Armstrong found a canonical model in the sense of the minimal model programme for the moduli space of principally polarised abelian varieties of dimension $g$ for $g\geq 12$. Namely they showed that the perfect cone  (or 1st Voronoi) compactification $\overline{A_g}^{p}$ has canonical singularities and that its canonical bundle is ample.
It is natural to ask whether a canonical model can be found as well for quotients of type IV domains by the action of an arithmetic group.
The analogue of the perfect cone compactification seems to be a good candidate. Let $L$ be an even lattice of signature $(2,n)$ and $D_{L}$ the associated hermitian symmetric domain of type IV. We consider the quotient $\cF_L(\widetilde{O}^+(L)):= D_{L}/\widetilde{O}^+(L))$ and we study the singularities of its perfect cone compactification $\overline{\cF_L(\widetilde{O}^+(L))}^{p}$. 
Under certain assumptions on the reflections in $\widetilde{O}^+(L)$ that fix a primitive isotropic vector in $L$ we are able to prove that the pair $(\overline{\cF_L(\widetilde{O}^+(L))}^{p}, \frac{1}{2}\Delta)$ has klt singularities, where $\Delta$ is the closure of the branch divisor of $D_{L}\to \cF_L(\widetilde{O}^+(L))$.
We recall that klt singularities form an important class in the study of higher dimensional algebraic varieties, in particular if $(X,D)$ is a klt pair, then $X$ has rational singularities. %Theorem 11.1.2 \cite{Kol96}
In section \ref{mod.sp.K3} we apply our result to the perfect cone compactification of the moduli space of $2d$-polarised $K3$ surfaces with ADE singularities when $d$ is square-free.

\section{Compactifications} 
Finding compactifications of locally symmetric varieties is a classical problem in geometry which has been extensively studied. We start by recalling the construction of the Baily-Borel and of the toroidal compactifications with the focus on the setting of our main theorem. Our main references are \cite{AMRT} and \cite{GHS13}.

\subsection{Baily-Borel and toroidal compactifications} \label{comps}
The Baily-Borel compactification and toroidal compactifications are all defined in general for quotients of Hermitian symmetric spaces by the action of an arithmetic subgroup of their automorphism group.
We treat quotients of type IV domains.
Let $L$ be an even lattice of signature $(2,n)$ with $n\geq 1$ and let 
\begin{align*}
D_L=\{[x]\in\bbP(L\otimes \bbC)\ |\ (x,x)=0\ (x,\bar{x})>0 \}^+
\end{align*}
(where ``+'' denotes a connected component) be the associated Hermitian symmetric domain of type IV.
Let $O^+(L)$ denote the subgroup of $O(L)$ preserving $D_L$ and let $D_L^{\bullet}$ denote the affine cone over $D_L$. We consider modular forms:

\begin{definition}
Let $n\geq 3$, $k\in \bbZ$, $\Gamma< O^+(L)$ be a subgroup of finite index and
$\chi:\Gamma\to\bbC^*$ be a character. A holomorphic function $F:D_L^{\bullet}\to\bbC$ is called a modular form of weight $k$ and character $\chi$ for the group $\Gamma$ if:
\begin{align*}
F(tZ)=t^{-k}F(Z)\ \forall t\in\bbC^*\\
F(gZ)= \chi(g)F(Z)\ \forall g\in \Gamma.
\end{align*}
\end{definition}
In the case $n<3$ an analogous definition would not imply the holomorphicity of $F$ at the boundary, for more details we refer to Definition 6.4 and its discussion in \cite{GHS13}.
Let $M_k(\Gamma, 1$) denote the space of weight $k$ modular form with trivial character. The Baily-Borel compactification of the quotient $\cF_L(\Gamma):=D_L/\Gamma$ can be defined as:
\begin{align*}
\overline{D_L/\Gamma}^{BB}:=\mathrm{Proj}\left(\oplus_k M_k(\Gamma, 1)\right).
\end{align*}

The Baily-Borel compactification is set-theoretically well understood, indeed the following holds:
\begin{theorem}\label{BBforFLGamma}(\cite{GHS13} Theorem 5.5)
The Baily-Borel compactification $\overline{\cF_L(\Gamma)}^{BB}$ decomposes as:
\begin{align*}
\overline{\cF_L(\Gamma)}^{BB} = \cF_L(\Gamma)\sqcup \bigsqcup_{\pi} X_{\pi}\sqcup\bigsqcup_{l}Q_l
\end{align*}
where $l$ and $\pi$ run through representatives of the finitely many $\Gamma$-orbits of isotropic lines and isotropic planes in $L\otimes\bbQ$ respectively. Each $X_{\pi}$ is a modular curve, each $Q_l$ is a point, and $Q_l$ is contained in the closure of $X_{\pi}$ if and only if the representatives may be chosen so that $l\subset \pi$.
\end{theorem}

The curves $X_{\pi}$ and the points $Q_l$ are called cusps.
We introduce now toroidal compactifications. In comparison with the Baily-Borel compactification they have the advantage that the boundary consists of a divisor; on the other hand they are not unique: they depend on the choice of certain fans.

Let $F$ be a cusp and $f\subseteq L\otimes \bbQ$ an isotropic space in the $\Gamma$-orbit corresponding to $F$. To every such $f$ we associate the stabiliser group $N(f)\subseteq O^+(L_{\bbR})$ of the isotropic space $f$ and the centre of the unipotent radical $U(f)\subseteq N(f)$. The group $U(f)$ has the structure of a vector space and $U(f)\cap\Gamma$ is a lattice in it. For every such $f$ we will consider in the next section an open convex cone $C(f)\subseteq U(f)$ (see Theorem III.4.1 in \cite{AMRT} for the definition of $C(f)$). Let $\overline{C(f)}^r$ be the union of $C(f)$ and the rational rays lying in its closure.
For every representative $f$ of a cusp let $\Sigma(f)$ be a fan whose support is  $\overline{C(f)}^r\subseteq U(f)$ and which is invariant under the action of the group $N(f)_{\Gamma}:=N(f)\cap\Gamma\subseteq O^+(L)$; we denote by $X_{\Sigma(f)}$ the associated toric variety. The domain $D_L$ can be embedded in an ambient space $D(f)$ and via this embedding we consider the compactification $\left(D_L/(U(f)\cap\Gamma)\right)_{\Sigma(f)}$ of the quotient $D_L/(U(f)\cap\Gamma)$ inside an $X_{\Sigma(f)}$-fibre bundle over $F$ (a toric variety in the case $F$ is a 0-dimensional cusp $Q_l$).
If such toric varieties are given by a $\Gamma$-admissible collection of polyhedra $\Sigma$ in the sense of Definition III.5.1 of \cite{AMRT}, their quotients by the induced action of $\Gamma$ glue to give a compactification of $\cF_L(\Gamma)$. In greater detail, let $\widetilde{D_L}^{\Gamma}=\sqcup_f \left(D_L/(U(f)\cap\Gamma)\right)_{\Sigma(f)}$ denote the union of the partial compactifications defined by the admissible collection of  polyhedra $\Sigma$, where the union is taken over the set of all representatives $f$ of cusps.
Then the toroidal compactification $\overline{\cF_L(\Gamma)}^{\Sigma}$ is the quotient of $\widetilde{D_L}^{\Gamma}$ by a certain equivalence relation $\Lambda$ induced by $\Gamma$. The definition and analysis of the relation $\Lambda$ is the subject of section III.6 of \cite{AMRT}. Thanks to Theorem III.5.2 of \cite{AMRT} the quotient $  \overline{\cF_L(\Gamma)}^{\Sigma}$ has the structure of a complex algebraic space.
The compactification $\overline{\cF_L(\Gamma)}^{\Sigma}$ comes with a surjection onto the Baily Borel compactification and analytic neighbourhoods of points lying over 0-cusps can be described as follows. Let $l\subseteq L$ be an isotropic line and $Q_{l}$ be the associated cusp.
The group $N(l)$ acts on the toric variety $X_{\Sigma(l)}$ preserving the partial compactification of $D_{L}/(U(l)\cap\Gamma)$. The quotient of the latter is isomorphic to an analytic open subset of the toroidal compactification $\overline{\cF_L(\Gamma)}^{\Sigma}$ containing the points lying over the 0-cusp $Q_{l}$.

When $\Gamma$ does not act freely it is meaningful to consider a neat normal subgroup of finite index: indeed given such a subgroup $\Gamma'\unlhd \Gamma$, a $\Gamma$-admissible collection of polyhedra $\Sigma$ is also $\Gamma'$-admissible. The finite group $\Gamma/\Gamma'$ acts on the associated toroidal compactification $\overline{\cF_L(\Gamma')}^{\Sigma}$ and has quotient $ \overline{\cF_L(\Gamma)}^{\Sigma}$.

Our main result concerns the quotient of $D_L$ by the action of a particular group that we now introduce. Recall that a non-degenerate lattice $\Lambda$ naturally injects in its dual $\Lambda^{\vee}:=\mathrm{Hom}_{\bbZ}(\Lambda, \bbZ)$ and that the quotient $A_{\Lambda}:=\Lambda^{\vee}/\Lambda$ is a finite group. The group $A_{\Lambda}$ is called the discriminant group. The pairing of $\Lambda$ induces a pairing on the dual with values in $\bbQ$ which in turn gives a $\bbQ/\bbZ$-valued form on $A_{\Lambda}$. If $\Lambda$ is even the form on $\Lambda^{\vee}$ yields:
\begin{align*}
q:A_{\Lambda}\to \bbQ/2\bbZ
\end{align*}
We denote by O$(A_{\Lambda})$ the group of automorphisms of $A_{\Lambda}$ preserving this form.
Our main result concerns the group
\begin{align*}
\widetilde{O}^+(L):=\ker\left( O^+(L)\to O(L^{\vee}/L) \right).
\end{align*}
We recollect some results contained in \cite{Ma2018}, Appendix A:
\begin{proposition}\label{MaAppendix}
\begin{sloppypar}
Given an isotropic line $l\subset L$, there is an isomorphism of lattices:
\begin{align*}
U(l)\cap \widetilde{O}^+(L) \simeq l^{\perp}/l.
\end{align*}

The action of $N(l)$ on the variety $X_{\Sigma(l)}$ factors via the quotient ${ \overline{N(l)}:=\frac{N(l)\cap\Gamma}{U(l)\cap\Gamma} }$, which is a subgroup of the orthogonal group O$(l^{\perp}/l)$.
\end{sloppypar}
\end{proposition}

\subsection{The perfect cone compactification} \label{perfcomp}
In this paragraph we introduce the toroidal compactification that is the object of our main result, the perfect cone compactification.

A cone in a real finite-dimensional vector space $V$ is a set closed under multiplication by a positive scalar. Given an open cone $C\subseteq V\setminus\{0\}$ whose closure does not contain any line, we denote by $\overline{C}^*\subseteq V^*$ the linear forms that are non-negative on $C$; we define the dual cone $C^*$ to be the interior of $\overline{C}^*\setminus \{0\}$. We say that a cone is self-adjoint if there is a positive definite form on $V$ whose induced isomorphism between $V$ and $V^*$ transforms $C$ in $C^*$. A self-adjoint cone is necessarily convex. We say that $C$ is homogeneous if the group of automorphisms of $V$ preserving the cone $C$ acts transitively on $C$.

We assume that $V\simeq L\otimes \bbR$ for a finitely generated free abelian group $L$, in particular $V$ has a rational structure. Let $G$ be an arithmetic subgroup of the automorphism group $\mathrm{Aut}(V,C)$ of $V$ fixing the cone $C$. Recall that a rational polyhedral cone $\sigma\subseteq V$ is a convex set of the form $\sigma=\bbR_{\geq 0}v_1+\bbR_{\geq 0}v_2+...+\bbR_{\geq 0}v_n$ for vectors $v_1,...,v_n\in L\otimes \bbQ$.

Defining a toroidal compactification reduces to the problem of finding fans for certain homogeneous cones $C$ with an action of $G$ which are $G$-admissible polyhedral decompositions in the sense of the following

\begin{definition} \label{admissible decomposition}
Given a self-adjoint homogeneous cone $C$ in a real vector space $V$ with integral structure $L\subseteq V$ and $G\subseteq\mathrm{Aut}(V,C)$ an arithmetic subgroup,
a $G$-admissible polyhedral decomposition for the cone $C$ is a family $S=\{\sigma_{\alpha}\}$ of rational polyhedral cones $\sigma_{\alpha}\subseteq \overline{C}$ such that:
\begin{itemize}
\item the family $S$ is closed under taking faces, i.e. for every $\sigma_{\alpha}\in S$, every face of $\sigma_{\alpha}$ is a cone in $S$;

\item for every pair of cones $\sigma_{\alpha},\ \sigma_{\beta}\in S$ the intersection is a face of both;

\item the group $G$ acts on $S$, i.e. for every element $g\in G$ and for every cone $\sigma_{\alpha}\in S$, $g\sigma_{\alpha}$ is a cone in $S$;

\item the family $S$ is finite modulo the action of the group $G$;

\item $C=\cup_{\alpha}(\sigma_{\alpha}\cap C)$.
\end{itemize}
\end{definition}

\begin{remark}
The reader may wonder what the relation between the support of $S$ and the rational closure of the cone $C$ is. When they coincide, the toroidal space defined by $S$ is compact.

On the one hand, if $C$ is a cone and $S$ a $G$-admissible polyhedral decomposition as in Definition~\ref{admissible decomposition}, then their union is contained in the rational closure $\overline{C}^r$, because all the cones in $S$ are rationally polyhedral.

On the other hand, given the quotient $\cF_L(\Gamma):=D_L/\Gamma$ and a $\Gamma$-admissible collection of polyhedra $\Sigma$, consisting of a $\Gamma(f)$-admissible polyhedral decomposition of $C(f)$, one for each isotropic space $f$, then the associated toroidal space $\overline{\cF_L(\Gamma)}^{\Sigma}$ is compact if and only if the equality 
\begin{align*}
\bigcup_{\Sigma(f)}\sigma = \overline{C(f)}^{r}
\end{align*}
holds for every $f$.
\end{remark}

Finally we recall Corollary II.5.23 of \cite{AMRT}
\begin{cor}
Let  $C$ be a self-adjoint homogeneous cone in a real vector space $V$ with integral structure $L\subseteq V$ and let $G\subseteq\mathrm{Aut}(V,C)$ be an arithmetic subgroup.
Taking cones over the faces of the closed convex hull of $\overline{C}\cap L\setminus\{0\}$ yields a $G$-admissible polyhedral decomposition of $C$.
\end{cor}

\begin{definition}
Given a self-adjoint homogeneous cone $C$ in a real vector space $V$ with integral structure $L\subseteq V$, we refer to the decomposition given above as the perfect cone decomposition $\Sigma_{p}(C)$.
\end{definition}

We define in the following the perfect cone compactification of $\cF_L(\Gamma)$.
Thanks to Lemma 2.25 of \cite{GHS07} for every 1 dimensional cusp $X_{\pi}$ the vector space $U(\pi)$ has dimension 1, therefore the definition of a toroidal compactification does not involve choices for the fans in $C(\pi)$.

The definition of a toroidal compactification for $\cF_L(\Gamma)$ then boils down to choosing an $N(l)_{\Gamma}$-admissible decomposition of $C(l)$ for every $\Gamma$-orbit of isotropic vectors $l\in L$. We come to the following

\begin{definition}
The perfect cone compactification $\overline{\cF_L(\Gamma)}^{p} $ is the toroidal compactification corresponding to the $\Gamma$-admissible collection of polyhedra $\Sigma_p$ consisting of the perfect cone decomposition $\Sigma_{p}(l)$ of $C(l)$ for every cusp $Q_l$ of $\overline{\cF_L(\Gamma)}^{BB}$.
\end{definition}

\begin{remark}
Thanks to the criterion for projectivity stated in section IV.2.1 of \cite{AMRT} the perfect cone compactification is projective.
\end{remark}

\begin{definition}
We denote by $\Delta$ the closure in $\overline{\cF_L(\widetilde{O}^+(L))}^{p}$ of the branch divisor of $D_{L}\to \cF_L(\widetilde{O}^+(L))$.
\end{definition}

Recall that an element of finite order $r\in\mathrm{GL}(n,\bbC)$ is said to be a quasi-reflection if all but one of its eigenvalues are equal to 1 and the eigenspace associated to the eigenvalue different from 1 has dimension 1. It is said to be a reflection if it is a quasi-reflection with an eigenvalue equal to -1.

\begin{lemma}
If $r\in N(l)\cap \widetilde{O}^+(L)$ is a reflection, then its image $\bar{r}\in\overline{N(l)}$ is a reflection.
\end{lemma}

\begin{proof}
Since $r$ has order 2, it does not lie in the torsion free group $U(l)\cap \widetilde{O}^+(L)$ and $\bar{r}$ has order 2 too. Let $H\subseteq L$ be the hyperplane pointwise fixed by $r$, then $\bar{r}$ fixes $(H\cap l)/l$ and is therefore a reflection. 
\end{proof}

Finally we can state our main theorem:

\begin{theorem}\label{MainThm}
If every reflection $s\in \overline{N(l)}$ admits a lifting to a reflection in $N(l)\cap \widetilde{O}^+(L)$, then the pair $\left(   \overline{\cF_L(\widetilde{O}^+(L))}^{p}, \frac{1}{2}\Delta \right)$ has klt singularities.
\end{theorem}

Kawamata log terminal (klt) singularities are mild singularities that appear in the context of the minimal model programme, we refer to Definition~\ref{Def-klt} for the precise definition. The proof of Theorem~\ref{MainThm} relies on the results about singularities of quotients of toric varieties obtained in Section~\ref{QuotientsSings}, in particular see Proposition~\ref{klt-toric}.

\begin{remark}
In the next section we give a concrete example for which the theorem holds.
\end{remark}

\begin{proof}
\begin{sloppypar}
Thanks to \cite{Mat02}, Proposition 4-4-4 the notion of klt singularities is local in the analytic category, therefore we may prove the statement for an open covering in the analytic topology.

Let $\Sigma_{p}(1)$ be the fan formed by the 1-dimensional cones of $\Sigma_{p}$. The associated partial compactification $\overline{\cF_L(\widetilde{O}^+(L))}^{\Sigma_{p}(1)}$ is an open subset of $\overline{\cF_L(\widetilde{O}^+(L))}^{p}$ containing all the points lying over the 1-cusps. 
Given a neat normal subgroup $\Gamma'\leq \widetilde{O}^+(L)$ of finite index, then the partial compactification $\overline{\cF_L(\Gamma')}^{\Sigma_{p}(1)}$ is smooth. The finite group $\widetilde{O}^+(L)/\Gamma'$ acts on it with quotient ${ q\colon\overline{\cF_L(\Gamma')}^{\Sigma_{p}(1)} \to \overline{\cF_L(\widetilde{O}^+(L))}^{\Sigma_{p}(1)} }$. Thanks to Theorem 1.1 of \cite{GHS07} every component of the ramification divisor of $q$ has ramification index equal to 1 and the branch divisor coincides with the restriction of $\Delta$. Hence the pair $(\overline{\cF_L(\widetilde{O}^+(L))}^{\Sigma_{p}(1)}, \frac{1}{2}\Delta)$ has klt singularities thanks to Lemma 3.16 of \cite{Kol96}.

We are left to show that any point lying over a 0-cusp has a neighbourhood $U$ such that the pair $(U, \frac{1}{2}\Delta|_{U})$ has klt singularities.

Given a point $\bar{x}\in\overline{\cF_L(\widetilde{O}^+(L))}^{p}$ lying over the 0-cusp in $\overline{\cF_L(\widetilde{O}^+(L))}^{BB}$ associated to the isotropic vector in $l\in L$. Let $V$ denote the analytic open subset in the toric variety $X_{\Sigma_{p}(l)}$ giving the partial compactification of $D_{L}/(U(l)\cap\Gamma)$ whose quotient by $\overline{N(l)}$ is isomorphic to an open neighbourhood of $\bar{x}$.
Let $x\in V$ be a lifting of $\bar{x}$ and $G:=\mathrm{Stab}_{\overline{N(l)}}(x)$ be its stabiliser, then there exists a $G$-invariant open neighbourhood $U_{x}\subseteq V$ of $x$ such that its image in the quotient by $\overline{N(l)}$ is isomorphic to $U_{x}/G$.
Let $\sigma\in\Sigma(l)$ be the cone in whose torus orbit $x$ lies. We denote by $T_{\sigma}$ the associated toric variety.  It follows that ${ G\subseteq\mathrm{Stab}_{\overline{N(l)}}(\sigma) }$, that $G$ acts on both $(U_{x}\cap T_{\sigma})$ and $T_{\sigma}$ and thus $(U_{x}\cap T_{\sigma})/G\subseteq T_{\sigma}/G$ is open.
Let $R$ be the ramification divisor of the quotient morphism $\pi\colon T_{\sigma}\to T_{\sigma}/G$. Thanks to Proposition \ref{MaAppendix} the group $G\subseteq \mathrm{O}(l^{\perp}/l)$ acts via group morphisms on $T_{\sigma}$, hence by Proposition \ref{klt-toric} it follows that the pair $\left(  (U_{x}\cap T_{\sigma})/G,  \frac{1}{2}\pi(R)|_{(U_{x}\cap T_{\sigma})/G} \right)$ has klt singularities.

Finally, thanks to Theorem 2.1 of \cite{GHS07}, the divisor $\Delta$ is induced by reflections of $\widetilde{O}^+(L)$. Hence, if $U_{\bar{x}}\subseteq\overline{\cF_L(\widetilde{O}^+(L))}^{p}$ denotes the open neighbourhood of $\bar{x}$ isomorphic to $(U_{x}\cap T_{\sigma})/G$, because of the bijection between the reflections in $\overline{N(l)}$ and $N(l)\cap\widetilde{O}^+(L)$ we conclude that the pair $(U_{\bar{x}}, \frac{1}{2}\Delta|_{U_{\bar{x}}})$ has klt singularities and the claim is shown.
\end{sloppypar}
\end{proof}

\subsection{The moduli space of polarised K3 surfaces with ADE singularities} \label{mod.sp.K3}
We recall here the construction of the moduli space of polarised K3 surfaces with ADE singularities and explain how Theorem \ref{MainThm} can be applied to it.

A complex K3 surface is a compact connected complex manifold of dimension 2 such that
\begin{align*}
\Omega^2_X\cong \mathcal{O}_X \mbox{ and } H^1(X,\mathcal{O}_X)=0.
\end{align*}

\begin{sloppypar}
The group $H^2(X,\mathbb{Z})$, equipped with the intersection form, is an even unimodular lattice of rank 22 of signature $(3,19)$ isomorphic to ${ \mathbf{\Lambda_{K3}}:=U^{\oplus 3}\oplus E_8^{\oplus 2}(-1) }$ where $U$ is the hyperbolic lattice and $E_8$ is the positive definite lattice associated to the $E_8$ Dynkin diagram (see Proposition 1.3.5 of \cite{HuyBook}).
\end{sloppypar}

Let $L$ be a line bundle on a K3 surface $X$. Recall that $L$ is big and nef if and only if $(L.C)\geq 0$ for all closed curves $C\subset X$, and $L^2>0$ (see for example Proposition 2.61 of \cite{KM}).
A quasi-polarised K3 surface is a pair $(X,L)$ where $X$ is a K3 surface and $L$ is a primitive big and nef line bundle on it.

After we fixed an isomorphism of lattices $\varphi: H^2(X,\mathbb{Z})\stackrel{\sim}{\to} \Lambda_{K3}$, a polarisation of degree $2d$ (i.e. $L^2=2d$) gives an element $h=\varphi([L])\in\Lambda_{K3}$ whose orthogonal complement $h^{\perp}$ is isomorphic to $\mathbf{\Lambda_{2d}}:=U^{\oplus 2}\oplus E_8^{\oplus 2}(-1)\oplus\langle -2d\rangle$.
We consider the period domain:
\begin{align*}
\mathcal{D}_{2d}:=\{x\in\mathbb{P}({\Lambda_{2d}}\otimes\mathbb{C})|\ (x,x)=0 \mbox{ and } (x,\bar{x})>0\}^+
\end{align*}
where the sign ``$+$'' denotes a connected component.

Denote by
\begin{align*}
O(\Lambda_{K3},h):=\{g\in O(\Lambda_{K3})|g(h)=h\}\leq O(\Lambda_{K3})
\end{align*}
the subgroup of automorphisms of $\Lambda_{K3}$ fixing $h$. The coarse moduli space of quasi-polarised K3 surfaces $\mathcal{F}_{2d}$ is given by the quotient of $\mathcal{D}_{2d}$ by the action of $\widetilde{O}^+(\Lambda_{2d})$, i.e. the finite index subgroup of $O(\Lambda_{K3},h)$ fixing the connected component $\mathcal{D}_{2d}$.

The disadvantage of working with the moduli functor of quasi-polarised K3 surfaces is that it is not separated; therefore it is more convenient to treat the moduli functor of polarised K3 surfaces with ADE singularities, which is separated and whose coarse moduli space can be identified with $\mathcal{F}_{2d}$ (for more details see Section 5.1.4 and Remark 6.4.5 of \cite{HuyBook}).

The quotient $\mathcal{F}_{2d}$ is a quasi-projective variety of dimension 19. We can state the following:
\begin{theorem} \label{main}
Let $\overline{\mathcal{F}_{2d}}^{p}$ be the perfect cone compactification of the coarse moduli space of polarised K3 surfaces with ADE singularities and let $\Delta$ be the closure of the branch divisor of the morphism $\pi:D_{2d}\to \mathcal{F}_{2d}$.
If $d$ is square-free, the pair $(\overline{\mathcal{F}_{2d}}^{p}, \frac{1}{2}\Delta)$ has klt singularities.
\end{theorem}
\begin{proof}
Thanks to Corollary 1.5.2 of \cite{Nik}, there is an isomorphism $O(\Lambda_{K3},h)\simeq\widetilde{O}(\Lambda_{2d})$.
Thanks to \cite{Sca87} Theorem 4.0.1, if $d$ is square free, there is one $\widetilde{O}^{+}(\Lambda_{2d})$ orbit of primitive isotropic vectors. Thus, given $\ell\in\Lambda_{2d}$ primitive isotropic there exists an element $\ell'\in\Lambda_{2d}$ giving a splitting $\Lambda_{2d}\simeq U\oplus \left(\ell^{\perp}/(\bbZ\ell)\right)$ with $\langle\ell,\ell'\rangle\simeq U$. Hence the assumptions of Theorem~\ref{MainThm} are fulfilled and the claim follows from Theorem~\ref{MainThm}.
\end{proof}

\section{Canonical singularities}
In this section we review some classes of singularities.
We review firstly the definition of canonical singularities and we recall a characterisation of canonical singularities for toric varieties that will be used in our exposition. We are also interested in a particular class of singularities of pairs, namely the Kawamata log terminal singularities.

\begin{definition}
We say that a normal $\bbQ$-Gorenstein variety $X$ over $\bbC$ has canonical % (respectively terminal)
singularities if there exists a proper birational morphism $f:X'\to X$ from a smooth variety $X'$ such that:
\begin{align*}
K_{X'} = f^*K_X + \sum_{i\in I}a_i E_i
\end{align*}
with $a_i\in\bbQ_{\geq 0}$ %(respectively $a_i\in\bbQ_{}>0$)
where $\{E_i\}_{i\in I}$ is the set of all irreducible prime divisors lying in the exceptional locus of $f$.
\end{definition}
If the condition $a_{i}\geq 0$ holds for a resolution, then for any proper birational morphism $g\colon V\to X$ with $V$ smooth all the coefficients of the exceptional divisors in the ramification formula $K_{V} = g^{*}K_{X} + \sum_{i}b_{i}E_{i}$ are non-negative.

Toric varieties with canonical singularities can be characterised in terms of the cones defining them.

Let $N$ be a finitely generated free abelian group and let $T_N$ be the associated torus whose lattice of 1-parameter subgroups is given by $N$; we denote by $M$ the lattice of characters of $T_N$ and by $\langle\ ,\rangle: M\times N\to \bbZ$ the natural pairing.
Let $\sigma \subset N\otimes \bbR$ be a rational convex polyhedral cone and let $U_{\sigma}$ be the associated affine toric variety. If $\rho\in \sigma(1)$ is a ray, we denote by $u_{\rho}$ the unique non-trivial primitive vector in $\rho\cap N$.
We consider the polytope $\Pi_{\sigma}:=$convex-hull$(\{0\}\cup\{u_{\rho}| \rho\in\sigma(1)\}) \subset N\otimes\bbR$.

\begin{proposition}[Proposition 11.4.12 of \cite{CLS}] \label{toric}
With the above notation,
\begin{itemize}
        \item the following are equivalent:
        \begin{itemize}
                
                \item the variety $U_{\sigma}$ is $\bbQ$-Gorenstein;
                
                \item there exists an $m\in M\otimes\bbQ$ such that $\forall \rho\in\sigma(1)$ $\langle m,u_{\rho}\rangle = 1$;
                
                \item the polytope $\Pi_{\sigma}$ has a unique facet $F$ not containing the origin.
                
        \end{itemize}

        \item If the variety $U_{\sigma}$ is $\bbQ$-Gorenstein, then the variety $U_{\sigma}$ has canonical singularities if and only if the only non-zero lattice points of $\Pi_{\sigma}$ lie in the facet not containing the origin (or, more concisely, $\Pi_{\sigma}\cap N\setminus \{0\} \subset F\cap N$).
\end{itemize}
\end{proposition}

\begin{remark}\label{redtomax}
If $\Sigma=\{\sigma_{\alpha}\}_{\alpha}$ is a fan and $X_{\Sigma}$ is the associated toric variety, then $X_{\Sigma}$ has canonical singularities if and only if for every cone $\sigma\in\Sigma$ the variety $U_{\sigma}$ has canonical singularities.

We can further restrict ourselves to consider only maximal cones, in the sense that the variety $X_{\Sigma}$ is $\bbQ$-Gorenstein (respectively has canonical singularities) if and only if for every maximal cone $\sigma\in\Sigma$ the variety $U_{\sigma}$ is $\bbQ$-Gorenstein (respectively has canonical singularities).
\end{remark}

We focus now on the toric varieties appearing in the definition of the perfect cone compactification.
Let $L$ be a finitely generated free abelian group and $C\subset L\otimes \bbR$ an open convex cone that is self-adjoint.
We consider the set $K_{p}$=convex-hull$(\overline{C}\cap L\setminus\{0\})$ and we denote by $\Sigma_{p}$ the fan given by taking cones over the faces of $K_{p}$ and the trivial cone, that is:
\begin{align*}
\Sigma_{p}=\{\sigma \subset L_{\bbR} \mbox{ cone } | \mbox{ there is a face } F\subset K_{p} \mbox{ such that } \sigma=\bbR_{\geq 0} F\}.
\end{align*}

\begin{theorem} \label{toric varieties canonical}
The toric variety $X_{\Sigma_{p}}$ associated to the fan $\Sigma_{p}$ is $\bbQ$-Gorenstein and has canonical singularities.
\end{theorem}

\begin{proof}
Thanks to Remark \ref{redtomax}, in order to prove the theorem it is enough to show that for every $\sigma\in(\Sigma_{p})_{max}$ the variety $T_{\sigma}$ is $\bbQ$-Gorenstein and has canonical singularities. Let $\sigma=\bbR_{\geq0}F$ be such a cone, where $F$ is a facet of $K_{p}$.
Let $H$ denote the supporting hyperplane of $F$, i.e. $F=H\cap K_{p}$, and let $H^-$ be the closed half-space such that $H^{-}\cap K_{p} = F$. The origin lies in the interior of $H^{-}$: since $\sigma$ is of maximal dimension, the origin does not lie on $H$. On the other hand, if $0$ is in the complement of  $H^{-}$, then $ \bbR_{>1}F\subseteq H^{-} $ and $ \bbR_{>1}F\subseteq \bbR_{>1}K_{p}\subseteq K_{p} $, which is absurd.

Recall the definition of the polytope $\Pi_{\sigma}:=$convex-hull$(\{0\}\cup\{u_{\rho}| \rho\in\sigma(1)\})$ and let $0\not=v\in L\cap\Pi_{\sigma} \subseteq L\cap \sigma \cap H^-$; there exists an $n\in \mathbb{N}$ such that $nv\in K_{p}\subseteq \overline{C}$.
Since $\overline{C}$ is a cone it follows that $v\in\overline{C}$; therefore $v\in L\cap\overline{C}\setminus\{0\}\subseteq K_{p}$. We conclude that $v\in K_{p}\cap H^-=F$ and, thanks to Proposition \ref{toric}, that $T_{\sigma}$ is $\bbQ$-Gorenstein and has canonical singularities.
\end{proof}

%%%%%%%%%%%%%%%%%%%%%%%%%%%%%%%%%%%%%%%%%%%%%%%%
%		QUOTIENTS OF TORIC VARIETIES
%%%%%%%%%%%%%%%%%%%%%%%%%%%%%%%%%%%%%%%%%%%%%%%

\section{Singularities of quotients of toric varieties}\label{QuotientsSings}
In this section we analyse quotients of toric varieties with canonical singularities by the action of finite groups acting via group homomorphisms of the torus.

The following is essentially \cite{GHS07}, Corollary~2.20. We include a proof to highlight the fact that the assumption of smoothness is not necessary.

\begin{lemma} \label{lemmaKlaus}
Let $N$ be a free abelian group and $\Sigma$ a fan of strongly convex rational polyhedral cones contained in $N_{\bbR}$; let $T\hookrightarrow X_{\Sigma}$ be the corresponding toric variety. Let $G\leq\mathrm{GL}(N)$ be a finite group acting on the variety $X_{\Sigma}$. Let $D$ be a divisor which is pointwise fixed by a nontrivial element $1\not=g\in G$, then $D$ is not a torus-invariant divisor, i.e. $D\cap T\not=\emptyset$.
\end{lemma}
\begin{proof}
\begin{sloppypar}
Looking for a contradiction, we assume that $D$ lies in the boundary. 

Let $M$ denote the dual of $N$. We can identify the torus $T$ with the group homomorphisms $T=\mathrm{Hom}(M,\bbC^*)\cong N\otimes \bbC^*$ and $D$ corresponds to a 1-dimensional cone in $\Sigma$, say $\rho$; in particular $D=\overline{(O(\rho))}$ where ${ O(\rho)=\{u:\rho^{\perp}\cap M\to \bbC^*| u }$ group homomorphism$\}$ is a torus orbit. Since $g$ acts trivially on the divisor $D$, it acts trivially on $\rho^{\perp}\cap M\cong \frac{N}{\rho\cap N}$. Therefore, we get the following short exact sequence of $\langle g\rangle$-equivariant group morphisms:
\begin{align*}
0\to \bbZ(\rho\cap N) \to N \to \frac{N}{\rho\cap N}\to 0.
\end{align*}
The action of $g$ on $\bbZ(\rho\cap N)$ can be just $\pm id$; nonetheless, since $g$ acts on $O(\rho)$, it acts on the toric affine variety $U_{\rho}$ associated to the cone $\rho$, hence $g$ acts as the identity on $\bbZ(\rho\cap N)$.  It follows that $g$ acts as the identity on $N$, which is a contradiction and the claim is shown.
\end{sloppypar}
\end{proof}

\begin{lemma}\label{ramification}
Let $T$ be a torus with associated 1-parameter subgroup lattice $N$. Let $T\hookrightarrow X$ be a toric variety acted on by a finite group $G<GL(N)$.
Then the ramification divisor of the quotient morphism $\pi\colon X\to X/G$ is given by
\begin{align*}
R(\pi) = \sum_{\substack{ r\in G \\ \textit{ reflection}} }  \mathrm{Fix}(r)
\end{align*}
where $\mathrm{Fix}(r)$ denotes the fixed locus of $r$ equipped with the reduced structure.
\end{lemma}

\begin{proof}
By Lemma \ref{lemmaKlaus}, no component of the ramification divisor is torus-invariant, hence if $R(\pi)$ has support on a prime divisor $D$, then $D = \overline{D\cap T}$ and we are left to prove the statement for the restriction $\pi|_{T}:T\to T/G$.

For this, we assume that $X= T$ is the torus and we consider a point $x\in D$ for a prime divisor $D$ fixed by an element $g\in G$. Multiplication by $x$ induces a $g$-equivariant morphism between the tangent spaces ${ m_{x}:T_{e}T\to T_{x}T }$. The tangent space to $D$ at $x$ is a $g$-invariant codimension 1 subspace $T_{x}D\subseteq T_{x}T$ and so is its preimage $m_{x}^{-1}(T_{x}D)\subseteq T_{e}T\simeq N\otimes\bbC$.
In particular, $g$ acts as the identity on $m_{x}^{-1}(T_{x}D)$, and is a quasi-reflection.
Let $\lambda$ be the unique non-trivial eigenvalue of $g$, since $g\in\mathrm{GL}(N)$, its trace $tr(g) = n-1+\lambda$ is an integer, so $g$ is a reflection and the claim is shown.
\end{proof}

As the next example shows, the quotient of a toric variety with canonical singularities by a finite group acting via automorphisms of the one parameter subgroup lattice need not be $\bbQ$-Gorenstein. In particular it seems that without more precise knowledge of the cones of the fan $\Sigma$ it is not possible to conclude that the perfect cone compactification  $\overline{\cF_L(\widetilde{O}^+(L))}^{p}$ has canonical singularities.

\begin{example}
Let $\sigma\subset \bbR^3 = N_{\bbR}$ be the cone over the polytope
\begin{align*}
P := \mbox{convex-hull}((-1,-1,1), (1,1,1), (0,2,1), (-1,1,1), (-1,-1,1)).
\end{align*}
The corresponding toric variety $T_{N} \hookrightarrow X_{\sigma}$ has, thanks to Proposition \ref{toric}, canonical singularities.
Let $r$ be the reflection fixing the coordinate plane $\{(0,y,z)\ :\ y,z\in \bbR\}$. The cone $\sigma$ is fixed by the action of $r$ on $N_{\bbR}$, therefore the group $G = \langle r\rangle$ generated by the reflection $r$ acts on the toric variety $X_{\sigma}$.
The pull-back formula in section 2.1 of \cite{Sho93} reads:
\begin{align*}
K_{X_{\sigma}} = f^*(K_{X_{\sigma}/G} + \frac{1}{2} f(R)) 
\end{align*}
where $R\in\mathrm{Div}(X_\sigma)$ is the ramification divisor of the quotient morphism $f\colon X_{\sigma}\to X_{\sigma}/G$.

Thanks to Corollary 2.2 of \cite{Sho93} it follows that $K_{X_{\sigma}/G}$ is $\bbQ$-Cartier if and only if $R$ is. Thanks to Proposition 4.1.2 of \cite{CLS} it can be computed that:
\begin{align*}
R = \mathrm{div}(x-x^{-1}) - \sum_{\rho_i\not= (0,2,1)} D_{\rho_i}
\end{align*}
where $D_{\rho_i}$ is the torus-invariant divisor corresponding to the ray $\rho_i\subseteq \sigma$. In particular neither $R$ nor $K_{X_{\sigma}/G}$  is $\bbQ$-Cartier.

The same argument carries over to many other polytopes that are symmetric with respect to an axis.
\end{example}

Even if the quotient might not be $\bbQ$-Gorenstein, we can prove a weaker statement about the singularities of quotients of toric varieties, for this we recall the notion of pair. We consider pairs $(X,B)$ where $X$ is a normal variety and $B:=\sum d_{i} D_{i}$ is a boundary, i.e. a $\bbQ$-divisor for which the coefficient of each prime divisor $D_{i}$ satisfies $0\leq d_{i}\leq1$.
If $K_{X}+B$ is $\bbQ$-Cartier, its pullback along any morphism is well-defined and for a birational morphism $f\colon Y\to X$ with $Y$ smooth, if $E_{i}$ denote the exceptional divisors, we have the equality
\begin{align*}
K_{Y}  = f^{*}(K_{X} + B) - f_{*}^{-1}(B) + \sum a_{i}E_{i}
\end{align*}
for some rational numbers $a_{i} = a(E_{i}, X, B)$. The coefficient $a_{i}$ is independent of the model $Y$, but it depends only on the centre of $E_{i}$ (\cite{Mat02}, Proposition-Definition 4-4-1), it is called discrepancy of $(X,B)$ at $E_{i}$.
We further define the discrepancy of $(X,B)$ to be:

\begin{flushright}
disc(X,B):=inf\{ $a(E,X,B)$: $E$ corresponds to a discrete valuation of $k(X)$ such that Center$_{X}(E)\not=\emptyset$ and codim$_{X}(\mathrm{Center}_{X}(E))\geq 2$\}.
\end{flushright}

\begin{definition}\label{Def-klt}
The pair $(X,B)$ with $B=\sum d_{i}D_{i}$ is said to have
\begin{itemize}
\item canonical singularities if $K_{X}+B$ is $\bbQ$-Cartier and disc$(X,B)\geq 0$;

\item Kawamata log terminal (klt) singularities if $K_{X}+B$ is $\bbQ$-Cartier, $d_{i}<1$ for any $i$ and disc$(X,B)>-1$.

\end{itemize}
\end{definition}

\begin{remark}
The variety $X$ has canonical singularities if and only if the pair $(X,0)$ has canonical singularities.
\end{remark}

\begin{proposition} \label{klt-toric}
Let $T$ be a torus with associated 1-parameter subgroup lattice $N$. Let $T\hookrightarrow X$ be a toric variety with canonical singularities acted on by a finite group $G<GL(N)$. If $\pi\colon X\to X/G$ denotes the quotient morphism and $R$ the ramification divisor, then the pair $(X/G, \frac{1}{2}\pi(R))$ has klt singularities.
\end{proposition}

\begin{proof}
Thanks to Lemma \ref{ramification} the ramification divisor equals $\sum_r D_r$, where the sum runs over all the reflections of $G$ and $D_r$ denotes the divisor pointwise fixed by the reflection $r$.
The claim now follows from Lemma 3.16 of \cite{Kol96} or from generalising word by word the proof of Corollary 2.2 of \cite{Sho93} to the case of klt singularities.
\end{proof}

\section*{Declarations}

\subsection*{Funding} The author was partially supported by the DFG through the research grant Le 3093/3-1

\subsection*{Conflicts of interest} The author declares no competing financial interests.

\subsection*{Data availability}
Data sharing not applicable to this article as no datasets were generated or analysed
during the current study.

\subsection*{Code availability} Not applicable.

\providecommand{\bysame}{\leavevmode\hbox to3em{\hrulefill}\thinspace}
\providecommand{\MR}{\relax\ifhmode\unskip\space\fi MR }
\providecommand{\MRhref}[2]{%
  \href{http://www.ams.org/mathscinet-getitem?mr=#1}{#2}
}
\providecommand{\href}[2]{#2}

%\bibliographystyle{amsplain}
%\bibliography{references}

\end{document}